\documentclass{amsart}
\usepackage[utf8]{inputenc}

\usepackage[english]{babel}

\usepackage{amssymb}
\usepackage{amsmath}
\usepackage{amsthm}
\usepackage{enumerate}
\usepackage{cite}
\usepackage{xcolor}

\title[Identifiability for a class of symmetric tensors]
{Identifiability for a class of symmetric tensors}
\date{}

\newcommand{\C}{\mathbb{C}}
\newcommand{\Z}{\mathbb{Z}}
\newcommand{\Pj}{\mathbb{P}}

\newcommand{\N}{\mathbb{N}}

\newtheorem{theorem}{Theorem}[section]
\newtheorem{proposition}[theorem]{Proposition}
\newtheorem{lemma}[theorem]{Lemma}

\theoremstyle{definition}
\newtheorem{definition}[theorem]{Definition}

\newtheorem{example}[theorem]{Example}
\newtheorem{remark}[theorem]{Remark}

\newcommand{\rank}{\operatorname{rank}}


\author[E.~Angelini]{Elena Angelini}
\address{Elena Angelini. Dipartimento di Ingegneria dell'Informazione e Scienze Matematiche, Universit\`a di Siena, Italy}
\email{elena.angelini@unisi.it}

\author[L.~Chiantini]{Luca Chiantini}
\address{Luca Chiantini. Dipartimento di Ingegneria dell'Informazione e Scienze Matematiche, Universit\`a di Siena, Italy}
\email{luca.chiantini@unisi.it}

\author[A.~Mazzon]{Andrea Mazzon}
\address{Andrea Mazzon. Dipartimento di Ingegneria dell'Informazione e Scienze Matematiche, Universit\`a di Siena, Italy}
\email{nnmazz@hotmail.it}

\begin{document}







\begin{abstract}
We use methods of algebraic geometry to find new, effective methods for detecting the identifiability of symmetric tensors. In particular,
for ternary symmetric tensors $T$ of degree $7$, we use the analysis of the Hilbert function
of a finite projective set, and the Cayley-Bacharach property, to prove that, when the Kruskal's ranks of a decomposition
of $T$ are maximal (a condition which holds outside a
Zariski closed set of measure 0), then the tensor $T$ is identifiable, i.e. the decomposition is unique, 
even if the rank lies beyond the range of application
of both the Kruskal's and the reshaped Kruskal's criteria.
\end{abstract}


\maketitle

\section{Introduction}
In the study of symmetric tensors $T$ of type $(n+1)\times\dots\times (n+1)$ ($d$ times), which can be identified with
polynomial forms of degree $d$ in $n+1$ variables, one of the main aspects concerns their Waring decompositions, i.e.
decompositions of type $T=L_1^d+\dots+L_r^d$, where each $L_i$ is a linear form. The (Waring) rank of $T$, which is the minimum
$r$ for which the decomposition exists, turns out to be a good measure for the complexity of $T$. Many recent results
are devoted to the computation of the (Waring) rank of forms. Effective methods for the computation of decompositions
are available, though their practical implementation is often out of reach (see e.g. \cite{HauensteinOedOttSomm},
\cite{Ange}). 

An important, open question, however, is related to the uniqueness of a decomposition with $r$ minimal, i.e.
the uniqueness of a decomposition which computes the rank. The existence of a unique decomposition that computes the
rank of $T$ is a property of $T$ denoted as {\it identifiability}. Uniqueness is important for the applications, at least for two reasons. 
First, the computation of an effective decomposition is often faster, when the target is uniquely determined. Second,
if we are looking for a specific decomposition, the uniqueness ensures us that if we get close to one of them, then that one
is exactly what we were looking for. Here, with {\it uniqueness} of the decomposition we mean uniqueness up
to trivialities, like permutation of the summands or their rescaling. We mention the papers \cite{AllmanMatiasRhodes09},
\cite{AnandkumarGeHsuKakadeTelgarsky14}, \cite{AppellofDavidson81} for examples of applications.

When $T$ is a sufficiently general form of rank $r$, and $rn+r+n+1<\binom{d+n}n$ (technically, when we are 
in the range of {\it subgeneric} ranks), then the identifiability of $T$ is known to hold, with the exclusion of a short list of 
cases for $(n,d,r)$, which are completely described (see \cite{COttVan17a}).

On the other hand, for specific forms $T$, the question of the computation of the rank and the identifiability of $T$ 
is still widely open, except for small values of the rank. More in details, given a form $T$, which arises either
from theoretical reasons or from experimental data, it is often possible to determine a specific decomposition
$T=L_1^d+\dots+L_r^d$, but it remains unclear whether the decomposition computes the rank or it is unique,
i.e. it is unclear if we can exclude the existence of a second decomposition with $r'\leq r$ summands.

For specific forms, of which we know a decomposition $T=L_1^d+\dots+L_r^d$, the most famous criterion
to exclude the existence of a second decomposition with  $r'\leq r$ summands goes back to the 70's, to the celebrated paper
by Kruskal \cite{Kruskal77}. The criterion, which is valid even for non-symmetric tensors but holds
only for small values of $r$, is based on the notion of {\it Kruskal's ranks} of a decomposition (see Definition \ref{krur} below): if
$r$ satisfies an inequality in terms of the Kruskal's ranks, then $r$ is the rank of $T$ and 
the decomposition is unique. Derksen \cite{Derksen13}  proved that the inequality is sharp:  
one can not hope to improve Kruskal's criterion by using just the Kruskal's ranks of the decomposition.

In the specific case of symmetric tensors, however, other methods have been introduced to study the 
uniqueness of a decomposition. One of them, based on the study of the kernels of catalecticant maps, 
can be found in \cite{MassaMellaStagliano18}. Another method, based on the Kruskal's original criterion, 
is described in \cite{COttVan17b} and it is based on the computation of the Kruskal's ranks of a {\it reshaping} 
of a decomposition, i.e. the Kruskal's ranks of the points obtained by raising the $L_i$'s to some
partial powers $d_1,d_2,d_3$, with $d_1+d_2+d_3=d$. Both methods usually work in a range wider than the range
in which the original Kruskal's criterion can be applied. Thus, e.g., the determination of the reshaped Kruskal's rank
can exclude Derksen's examples in the symmetric case (see \cite{BallC12}, \cite{BallC13} for other methods). Yet, all the previous methods will give no answer when $r$ grows enough. For instance, for the case of ternary septics, i.e. the case
$n=2$, $d=7$,  subgeneric ranks (for which identifiability holds
generically) are ranks $r=1,\dots, 11$, but the original Kruskal's criterion will give no answer
for $r>7$, while both the catalecticant method or the reshaped Kruskal's method will work only for $r\leq 10$.

The aim of this paper is an extension of a method, introduced in \cite{COttVan17b} and  in \cite{AngeCVan18}
for the case of forms of degree $4$, which can determine that $r$ is the rank of $T$ and the decomposition is unique,
even in a range in which both the catalecticant method or the reshaped Kruskal's method do not work.
The method is still based on the computation of the reshaped Kruskal's ranks of the decomposition,
but the conclusion is obtained with advanced tools of algebraic geometry: the analysis of the Hilbert function
of the set of projective points associated with a decomposition.

We apply the method to ternary forms of degree $7$, and we find that if $T$ has a decomposition
$T=L_1^d+\dots+L_r^d$ whose reshaped Kruskal's ranks are general, then $T$ has rank $r$ and the
decomposition is unique (up to trivialities). We obtain thus an effective criterion for the identifiability 
of specific symmetric tensors, which definitively improves our previous knowledge.
 The proof of the criterion is based on results for the analysis of the postulation of finite sets in
 projective spaces (e.g. on the Cayley-Bacharach property)
 which, we believe, have an independent theoretical interest for investigators in the field.
In particular, our result also proves that, for a general choice of the
linear forms $L_i$'s, the span of $L_1^d,\dots,L_r^d$, which is a subset of the $r$-th secant
variety $S$ of the Veronese image of degree $7$ of $\Pj^2=\Pj(\mathbb C^3)$, does not contain singular points
of $S$, outside those generated by a proper subset of the powers $L_i^d$.

In section \ref{cub} we consider forms of any degree, with a decomposition in which the Kruskal's ranks
are not maximal, but are maximal modulo the fact that the points representing the decomposition lie
in a plane cubic curve. We show also in this case that, with the same methods, one can prove the identifiability of the
form $T$.
 
 We end by noticing that  ternary septics $T$ are the last numerical case in which the analysis of the
 Kruskal's rank of a decomposition is sufficient to conclude, outside a (Zariski closed) set of measure $0$, that
 a given decomposition is the unique one that computes the rank of $T$. In a forthcoming series of papers we will prove that,
 already for ternary optics, the linear span of every general decomposition contains both identifiable and not identifiable points.

\vskip.3cm
\paragraph{\textbf{Acknowledgements.}} 
The first two authors are members of the Italian GNSAGA-INDAM and are supported by the Italian PRIN 2015 - 
Geometry of Algebraic Varieties (B16J15002000005).

\section{Notation}\label{nota} 

We work over the complex field $\mathbb C$.

For any finite subset $A$ of the projective space $\Pj^n=\Pj(\C^{n+1})$, we denote by $\ell(A)$ the cardinality of $A$.

We call $v_d:\Pj^n\to \Pj^N$, $N = \binom{n+d}d - 1 $,  the Veronese embedding of degree $d$.
The space $\Pj^N$ can be identified with $\Pj(Sym^d(\C^{n+1}))$. So,
the points of $\Pj^N$ can be identified, up to scalar multiplication, with forms $T$ of degree $d$ in $n+1$ variables.

 By abuse, we will denote by $T$ both the form in $Sym^d(\C^{n+1})$ and the point in $\Pj^N$ which represents $T$.
 
 The form $T$ belongs to the \emph{Veronese variety} $v_d(\Pj^n)$ if and only if $T=L^d$,
 for some linear form $L$.

\smallskip

With the above notations we give the following definitions.

\begin{definition}
Let $A \subset \Pj^n$ be a finite set, $A=\{P_1,\dots, P_r\}$. $A$ is a \emph{decomposition} of $T\in \Pj(Sym^d(\C^{n+1}))$ 
if $ T\in \langle v_d(A) \rangle$, 
the linear space spanned by the points of $v_d(A)$. In other words, for a choice of scalars $a_i$'s,
$$T = a_1v_d(P_1)+\dots + a_rv_d(P_r). $$
The number $\ell(A)$ is the \emph{length} of the decomposition.

The decomposition $A$ is \emph{minimal} or \emph{non-redundant} if $T$ is not contained in the span of $v_d(A')$, for any proper subset
$A'\subset A$. In particular, if $v_d(A)$ is not linearly independent, then $A$ can not be minimal.
\end{definition}

If we identify points $P\in\Pj^n=\Pj(\C^{n+1})$ with linear forms $L$ in $n+1$ variables, then the image $v_d(L)$ corresponds
to the power $L^d$, so that a decomposition of $T\in \Pj(Sym^d(\C^{n+1}))$ corresponds to a Waring decomposition of a form of degree $d$.

\begin{definition}\label{krur}
For a finite set $A \subset \Pj^n$, the \emph{Kruskal's rank} $k(A)$ of $A$ is the maximum $k$ for which any subset
of cardinality $\leq k$ of $A$ is linearly independent.

The Kruskal's rank $k(A)$ is bounded above by $n+1$ and $\ell(A)$. If $A$ is general enough,
then $k(A)=\min\{n+1, \ell(A)\}$. 
\end{definition}

For instance, if $A\subset \Pj^3$ is a set of cardinality $5$, with a subset of $4$ points on a plane and no
three points aligned, then $k(A)=3$.

\begin{remark} It is straightforward that $k(A)$ attains the maximum $\min\{n+1, \ell(A)\}$ when all subsets of $A$
cardinality at most $n+1$ are linearly independent. In this case, for any subset $A'\subset A$ one has
$k(A')=\min\{n+1, \ell(A')\}$.

Notice also that for any subset $A'\subset A$, we have $k_{A'}\geq \min\{\ell(A'), k_A\}$.
\end{remark}

We can use the Veronese maps to define the higher Kruskal's ranks of a finite set $A$.

\begin{definition}
For a finite set $A \subset \Pj^n$, the \emph{$d$-th Kruskal's rank} $k_d(A)$ of $A$ is the Kruskal's rank
of the image of $A$ in the Veronese map $v_d$.

Thus, the Kruskal's rank $k(A)$ coincides with the first Kruskal's rank $k_1(A)$.
\end{definition}

The $d$-th Kruskal's rank $k_d(A)$ is thus bounded by $\min\{\ell(A), \binom{n+d}n\}$.

\begin{remark}\label{ksub} Since the projective spaces are irreducible, then any subset of a general finite set $A$ is general.
Thus one can prove that, for a sufficiently general subset $A\subset\Pj^n$, then all the
Kruskal's ranks $k_d(A)$ are maximal.
\end{remark}

\section{Preliminary results} 

We collect in this section the main technical tools for the investigation of the identifiability
of symmetric tensors. For the proofs of most of them, we refer to Sect. 2 of \cite{C} and Sect. 2.2 of \cite{AngeCVan18}.

\begin{definition}
Let $Y\subset \C^{n+1} $ be an ordered, finite set of cardinality $\ell $ of vectors. Fix an integer $ d \in \N $. 

The \emph{evaluation map of degree $d$ on $Y$} is the linear map
$$ ev_{Y}(d): Sym^d(\C^{n+1}) \to \C^\ell $$ 
which sends $ F \in Sym^d(\C^{n+1}) $ to the evaluation of $ F$ at the vectors of $Y$. 

Let $Z \subset \Pj^n $ be a finite set. Choose a set of homogeneous coordinates for the points of $Z$.
We get an ordered set of vectors $Y\subset \C^{n+1} $, for which the evaluation map  $ev_{Y}(d)$  is defined for every $d$.

If we change the choice of the homogeneous coordinates for the points of the fixed set $Z$,  the evaluation map
changes, but all the evaluation maps have the same rank. 
So, we can define the  \emph{Hilbert function} of $Z$ as the map:
$$ h_Z : \Z \to \N \qquad h_Z(d) = \rank(ev_{Y}(d)) .$$ 
We point out that the Hilbert function does not depend on the choice of the coordinates, as well as it 
does not vary after a change of coordinates in $\Pj^n$.

We define the \emph{first difference of the Hilbert function} $Dh_Z$ of $Z$ as:
$$ Dh_Z(j) = h_Z(j)-h_Z(j-1),\quad j \in \Z .$$
\end{definition}

Let $v_d: \Pj^n\to \Pj^N$ be the $d$-th Veronese embedding of $\Pj^n$.
For any finite set $Z\subset \Pj^n$, and for any $d \geq 0 $, the value $ h_Z(d) $ determines the dimension 
of the span of $v_d(Z)$. Indeed we have the following straightforward fact:

\begin{proposition}\label{dimspan}
Let $v_d: \Pj^n\to \Pj^N$ be the $d$-th Veronese embedding of $\Pj^n$.
$$ h_Z(d) = \dim (\langle v_d(Z) \rangle) +1. $$
\end{proposition}

From the previous proposition it follows that the Kruskal's rank $k_d(Z)$ is maximal
if and only if every subset of cardinality at most $N+1=\binom{n+d}n$ of $v_d(Z)$ 
is linearly independent. 
\smallskip

Several properties of the Hilbert functions and their differences are well known
in Algebraic Geometry. We will need in particular the following facts (for the proofs, see Sect. 2 of \cite{C}).

\begin{proposition} \label{inclu} If $Z' \subset Z$, then for every $d$ 
we have $h_{Z'}(d) \leq h_Z(d) $ and $Dh_{Z'}(d) \leq Dh_Z(d).$
\end{proposition}

\begin{proposition}\label{nonincr}
Assume that  for some $j>0$ we have $Dh_Z(j) \leq j$. Then:
$$ Dh_Z(j) \geq Dh_Z(j+1). $$
In particular, if  for some $j>0$, $Dh_Z(j)=0 $, then $Dh_Z(i)=0$ for all $i\geq j$.
\end{proposition}

We introduce the following notation, which comes from a cohomological interpretation of the 
Hilbert function of $Z$. We will use it as a mere symbolism.

\begin{definition} For every finite set $Z$ and for every degree $d$, define:
$$h^1_Z(d)= \ell(Z)-h_Z(d) = \sum_{j=d+1}^\infty Dh_Z(j).$$
In particular, from Proposition \ref{dimspan} we get:
\begin{equation}\label{h1}
\dim (\langle v_d(Z)\rangle) = \ell(Z)-1- h^1_Z(d).
\end{equation}
\end{definition}

Often, we will use the properties of the Hilbert function of the union of two different decompositions 
of a form. Thus, for our application the following property, which is a consequence of a 
straightforward application of the Grassmann formula, has particular interest.

\begin{proposition}\label{cap} Let $A,B\subset \Pj^n$ be \emph{disjoint} finite sets and set $Z=A\cup B$.
Then for any $d$, the spans of $v_d(A)$ and $v_d(B)$ satisfy the following formula:
$$\dim(\langle v_d(A)\rangle\cap \langle v_d(B)\rangle)+ 1 = h^1_Z(d).$$
\end{proposition}
\begin{proof} See \cite{AngeCVan18}, Sect. 6.
\end{proof}

Next, we define the {\it Cayley-Bacharach} property for a finite set in a projective space.

\begin{definition}\label{CB}
A finite set $Z\subset \Pj^n$ satisfies the \emph{Cayley-Bacharach property in degree $d$}, 
abbreviated as $CB(d)$, if, for any $P \in Z$, 
 every form of degree $d$ vanishing at $ Z\setminus\{ P\}$
also vanishes at $P$.
\end{definition}

\begin{example} 
Let $Z$ be a set of $6$ points in $\Pj^2$. 

If the $6$ points are general, then $Dh_Z(i)=i+1$ for $i\in\{0,1,2\}$. Moreover  $Z$ satisfies $CB(1)$.
Since $h_Z(2)=6$, $Z$ does not satisfy $CB(2)$.

If $Z$ is a general subset of an irreducible conic, then $Dh_Z(2)=2$ and $Dh_Z(3)=1$,  
moreover $Z$ satisfies $CB(2)$, hence it satisfies $CB(1)$.
 
If $Z$ has $5$ points on a line plus one point off the line, then  $Dh_Z$ is given by the
following table:

\begin{center}\begin{tabular}{c|ccccccc}
$d$ & 0 & 1 & 2 & 3 & 4 & 5 &  \dots \\ \hline
$Dh_Z(d)$ & 1 & 2 & 1 & 1 & 1 & 0 &  \dots
\end{tabular} 
\end{center}
Moreover $Z$ does not satisfy $CB(1)$.
\end{example}

\begin{remark}\label{CBprop} 
If $Z$ satisfies $CB(d)$, then it satisfies $CB(d-1)$ too. Otherwise, one could find 
$P \in Z$ and a  hypersurface $F \subset \Pj^n$ of degree $d-1$ such that $Z \setminus \{P\} 
\subset F$ and $P \notin F$. Therefore, if $H\subset \Pj^n$ is a hyperplane missing
$P$, then $F \cup H$ contains $Z \setminus \{P\}$ and misses $P$,
contradicting the $CB(d)$ hypothesis. 
\end{remark}

The main reason to introduce the Cayley-Bacharach properties relies in the following observation.

\begin{proposition} \label{CBonHilb}
$Z$ satisfies $CB(d)$ if and only if for all $P\in Z$ and for all $j\leq  d$
we have $h_Z(j)=h_{Z\setminus\{P\}}(j)$. I.e. $Z$ satisfies $CB(d)$ if and only if
$$Dh_Z(j)=Dh_{Z\setminus\{P\}}(j) \quad \forall j\leq d.$$
\end{proposition}
\begin{proof} Fix coordinates for the points of $Z$ and consider the corresponding evaluation map
$\C[x_0,\dots,x_n]\to \C^{\ell(Z)}$ in degree $j$. The kernel is the space
of forms of degree $j$ vanishing on $Z$. Since $Z$ satisfies $CB(j)$ for all $j\leq d$,
then, for all $P\in Z$, every form of degree $j\leq d$ that vanishes on $Z\setminus\{P\}$
also vanishes on $Z$. Thus, in any degree $j\leq d$  the kernel of the evaluation map 
does not change if we forget any point $P\in Z$. Hence $h_Z(j)=h_{Z\setminus\{P\}}(j)$.
\end{proof}

\begin{remark}\label{CBcons} 
We will use the previous proposition as follows.

Assume that $Z$ does not satisfy $CB(d)$. Then there exists a points $P\in Z$ and an integer $j_0\leq d$ such that
$Dh_Z(j_0)>Dh_{Z\setminus\{P\}}(j_0)$. Since $\sum Dh_Z(i)=\ell(Z) =\ell(Z\setminus\{P\})+1 =
\sum Dh_{Z\setminus\{P\}}(i)$, it follows that $Dh_Z(i) = Dh_{Z\setminus\{P\}}(i)$ for all $i\neq j_0$.
In particular, the equality holds for $i\geq d+1$. Thus there exists $P\in Z$ such that 
$$h^1_Z(d) = \sum_{i=d+1}^\infty Dh_Z(i) = \sum_{i=d+1}^\infty Dh_{Z\setminus\{P\}}(i)= h^1_{Z\setminus\{P\}}(d).$$
\end{remark}

The following proposition, which gives a strong constraint on the Hilbert function of sets with a Cayley-Bacharach property,
 is a refinement of a result due to Geramita, Kreuzer, and Robbiano (see Corollary 3.7 part (b) and (c) of \cite{GerKreuzerRobbiano93}).

\begin{theorem}\label{GKRext}
If a finite set $ Z \subset \Pj^{n} $ satisfies $\mathit{CB}(i)$, then for any $ j $ such that $ 0 \leq j \leq i+1 $ we have
$$ Dh_{Z}(0)+Dh_{Z}(1)+\cdots + Dh_{Z}(j) \leq Dh_{Z}(i+1-j)+\cdots +Dh_{Z}(i+1).$$
\end{theorem}
\begin{proof} 
See Theorem 4.9 of \cite{AngeCVan18}.
\end{proof}

\section{Kruskal's criterion for symmetric tensors}

The symmetric version of the Kruskal's criterion for the identifiability of tensors can be found in 
\cite{COttVan17b}. We recall it  for the reader's convenience.

\begin{theorem} \label{K} ({\bf Kruskal's criterion}) Let $T$ be a form of degree $d$ and let $A$ be 
a minimal decomposition of $T$, with $\ell(A)=r$. 
Fix a partition $a,b,c$ of $d$ and call $k_a,k_b,k_c$ the Kruskal's ranks of $v_a(A),v_b(A), v_c(A)$ respectively.
If:
$$ r\leq \frac {k_a+k_b+k_c-2}2,$$
then $T$ has rank $r$ and it is identifiable.
\end{theorem}

\section{The extension}

In this section, we prove the main result, i.e. an extension of the Kruskal's criterion for the case of
septics in $3$ variables. 

Let $T$ be a septic and consider a minimal decomposition of $T$ with $11$ elements
$A=\{P_1,\dots,P_{11}\}$. We assume that the points of $A$ are general, and more precisely:
\begin{itemize}
\item the Kruskal's rank of $A$ is $3$, i.e. no three points of $A$ are aligned;
\item the Kruskal's rank of $v_3(A)$ is $10$, i.e. no $10$ points of $A$ are contained in a cubic or, equivalently, 
for any subset $A'\subset A$ with $\ell(A')\leq 10$, the set $v_3(A')$ is linearly independent.
\end{itemize}
With these assumptions, we get:

\begin{theorem} $T$ has rank $11$ and it is identifiable.
\end{theorem}

Notice that the original Kruskal's criterion \ref{K} does not cover this case. Namely if we take a partition $7=3+3+1$,
then under our hypothesis the given decomposition $A$ of $T$ satisfies $k_3=10$, $k_1=3$, but:
$$ 11 \not\leq \frac {10+10+3-2}2.$$
A similar computation holds for any other partition of $7$. Indeed  in section 4 of \cite{COttVan17b} it is pointed out
that the partition $7=3+3+1$ is the one that covers the widest range for the application of Kruskal's criterion.

We will prove the extension of the Kruskal's criterion by contradiction. So, from now on consider the set 
$Z=A\cup B$, where $B$ is a second \emph{minimal} decomposition of $T$ with at most $11$ points.
Since $Z$ contains $A$, then the difference $Dh_Z$ of the Hilbert function of $Z$  is:

\begin{center}\begin{tabular}{c|cccccccccc}
$j$ & $0$ & $1$ & $2$ &   $3$ &   $4$ &   $5$ &  $6$ &  $7$ &  $8$ &\dots \\  \hline
$Dh_Z(j)$ & $1$ & $2$ &   $3$ &   $4$ &   $a_4$ &  $a_5$ &  $a_6$ &  $a_7$ & $a_8$ & $\dots$
\end{tabular}\end{center}
with $a_8>0$, since $T\in \langle v_7(A)\rangle\cap \langle v_7(B)\rangle$ and $A,B$ are minimal, thus the points
of $v_7(Z)$ can not be linearly independent. It follows that $h^1_Z(7)>0$, as in
Proposition \ref{cap}.
Hence also $a_4, a_5, a_6, a_7>0$, by Proposition \ref{nonincr}. 

\begin{proposition}\label{CB7} $ Z $ can not satisfy $ CB(7) $.
\end{proposition}
\begin{proof} Assume that $Z$ satisfies $  CB(7) $. Then, by Proposition \ref{CBonHilb},
$a_5+a_6+a_7+a_8\geq 1+2+3+4=10$. Since $\ell(Z)\leq 22$,
this implies $a_4\leq 2$. But then, by Proposition \ref{nonincr}, $ a_5+a_6+a_7+a_8\leq 2+2+2+2=8$, a contradiction.
\end{proof}

Next, we need a result which is the application of the argument of Remark  \ref{CBcons} to the existence of two decompositions of
a tensor $T$. Since we will use it in several contests, we prove it in a general setting.

\begin{lemma}\label{CBdis}
Let $U$ be a symmetric tensor and consider two \emph{minimal} decompositions $A_1,A_2$ of $U$. 
Set $W= A_1 \cup A_2$. If $A_1\cap A_2 = \emptyset$,  then $W$ has the Cayley-Bacharach property $CB(d)$.
\end{lemma}
\begin{proof}
Suppose that $CB(d)$ does not hold for $W$. Hence there is a point $P \in W$ and a and a hypersurface $F\subset\Pj^n$
 of degree $d$ such that $F$ contains $W\setminus\{P\}$ but misses $P$. Just as in Remark \ref{CBcons},
we get that:
$$h^1_W(d)=h^1_{W\setminus\{P\}}(d).$$
We do not know in principle if $P$ belongs either to $A_1$ or to $A_2$. In any case, by Proposition \ref{cap}:
\begin{multline*} \dim \langle v_d(A_1)\rangle\cap \langle v_d(A_2)\rangle = h^1_W(d)-1 \\=h^1_{W\setminus\{P\}}(d) -1 =
\dim \langle v_d(A_1\setminus\{P\})\rangle\cap \langle v_d(A_2\setminus\{P\})\rangle.\end{multline*}
Thus $U$ belongs to both $v_d(A_1\setminus\{P\})$ and $v_d(A_2\setminus\{P\})$. If $P\in A_1$,
we get a contradiction with the minimality of $A_1$. Similarly, if $P\in A_2$, we get a 
contradiction with the minimality of $A_2$.
\end{proof}

As a consequence of Lemma \ref{CBdis} and Proposition \ref{CB7}, applied to the decompositions
$A,B$ of $T$, in our setting we find:

\begin{proposition} \label{capcap} The intersection  $A\cap B$ is non empty.
\end{proposition}

So we must have $\ell(A\cap B)=i>0$. After rearranging the points of $A,B$, we may assume
$$ A=\{P_1,\dots,P_i,P_{i+1},\dots,P_{11} \}\qquad B=\{P_1,\dots,P_i,P'_{i+1},\dots,P'_q\},$$
where $q=\ell(B)\leq 11$, $i<q$, and the set $B_0=\{P'_{i+1},\dots,P'_q\}$ is disjoint from $A$, i.e. $B_0=B\setminus A$. 
Then we know that $Z=A\cup B_0$ has not the property $CB(7)$.

For a choice of the scalars, we can write:
\begin{multline*}
 a_1v_7(P_1)+\dots + a_{11}v_7(P_{11}) = T  \\
= b_1v_7(P_1)+\dots +b_iv_7(P_i)+b_{i+1}v_7(P'_{i+1})+\dots +b_qv_7(P'_q).
\end{multline*}
The minimality of $A,B$ implies that none of the coefficients $a_i,b_i$ is $0$.

Write:
\begin{gather*} 
T_0= (a_1-b_1)v_7(P_1)+\dots + (a_i-b_i)v_7(P_i)+a_{i+1}v_7(P_{i+1})+\dots + a_{11}v_7(P_{11})\\
 = b_{i+1}v_7(P'_{i+1})+\dots +b_qv_7(P'_q)
 \end{gather*}
  so that $T=T_0+b_1v_7(P_1)+\dots +b_iv_7(P_i)$. 
 
 We are now ready to prove that:
 
 \begin{proposition} The existence of $B$ yields a contradiction.
 \end{proposition}
\begin{proof} Since $Z=A\cup B_0$ does not satisfy $CB(7)$, as in the proof of  Lemma \ref{CBdis},
we know that there exists a point $P\in Z$ such that 
$$\langle v_7(A\setminus\{P\})\rangle\cap \langle v_7(B_0\setminus\{P\})\rangle =
\langle v_7(A)\rangle\cap \langle v_7(B_0)\rangle,$$
hence $T_0$ belongs to $\langle v_7(A\setminus\{P\})\rangle\cap \langle v_7(B_0\setminus\{P\})\rangle$. If
$P$ belongs to $B_0$, then $T$ is spanned by $v_7(B\setminus \{P\})$, which contradicts the minimality of $B$.
Similarly, if $P=P_j$ with $j>i$, then $T$ is spanned by $v_7(A\setminus \{P\})$, which contradicts the minimality of $A$.
Thus $P$ is a point in $A\cap B$, and, after rearranging, we may assume $P=P_1$.

Thus $T_0=c_2v_7(P_2)+\dots + c_{11}v_7(P_{11})$. If $a_1-b_1\neq 0$, this implies that $v_7(A)$ is linearly
dependent, which contradicts the minimality of $A$. Thus $a_1=b_1$, so that $T_0$ is spanned by 
$v_7(P_2),\dots, v_7(P_{11})$. It may happen that some other coefficient $(a_i-b_i)$ is $0$.
Nevertheless $T_0$ has a minimal decomposition $A'\subset A$ whose length $\ell(A')$ satisfies by $11-i\leq\ell(A')\leq 10$, namely:
$$T_0= (a_2-b_2)v_7(P_2)+\dots + (a_i-b_i)v_7(P_i)+a_{i+1}v_7(P_{i+1})+\dots + a_{11}v_7(P_{11}).$$
Moreover $T_0$ has a decomposition, $B_0$, of length $\ell(B_0)=q-i\leq 11-i\leq \ell(A')$. Since $A'$ consists of at most $10$ points of $A$, then
$v_3(A')$ is linearly independent. Thus, the Kruskal's rank $k'_3$ of $v_3(A')$ is $\ell(A')$, while the Kruskal's rank $k'_1$ of $A'$ is 
$\min\{\ell(A'),3\}$. Moreover $\ell(A')>1$,  otherwise $i=10$, $q=11$ and $T_0=a_{11}v_7(P_{11})=b_{11}v_7(P'_{11})$, 
which means that, as projective points,  $P_{11}=P'_{11}$, a contradiction. Thus one computes:
$$ \ell(A')\leq \frac {k'_3+k'_3+k'_1-2}2.$$
It follows that the Kruskal's criterion contradicts the existence of the two different decompositions of $T_0$. 
\end{proof}

\begin{example} To give an example, consider the linear forms $L_1,\dots,L_{11}$ in $3$ variables associated to the points
\begin{gather*}
(1,0,0),\ (0,1,0),\ (0,0,1),\, (1,1,1),\ (1,-1,2),\ (1,3,-1),\\
(1,2,3),\ (2,-1,1),\ (-2,-1,3),\ (-1,3,4), \ (3,-1,4).
\end{gather*}
One computes easily that the set $A$ of the points has maximal Kruskal's ranks $k_1=3$ and $k_3=10$.
It turns out  that any linear combination of the 7th powers of the forms $L_i$'s, with no
zero coefficients, has rank $11$ and it is identifiable. 
For instance, this holds for the sum $T=L_1^7+\dots +L_{11}^7$, i.e.
{\small 
\begin{gather*} T=2191x^7-849x^6y+249x^5y^2-51x^4y^3+45x^3y^4+501x^2y^5+69xy^6+4500y^7+\\3181x^6z -918x^5yz
+430x^4y^2z-204x^3y^3z+346x^2y^4z-1128xy^5z+2390y^6z+3631x^5z^2\\-1390x^4yz^2+274x^3y^2z^2 +344x^2y3z^2
-1034xy^4z^2+4390y^5z^2+5731x^4z^3-1668x^3yz^3\\+1372x^2y^2z^3 -1686xy^3z^3 +5636y^4z^3+6115x^3z^4
-1714x^2yz^4-1346xy^2z^4\\ +7234y^3z^4+11491x^2z^5-5208xyz^5 +11480y^2z^5+7531xz^6+8860yz^6+37272z^7.
\end{gather*}}
\end{example}

\section{Decompositions on cubics}\label{cub}
In this section we will prove results on the identifiability of a particular set of  tensors.
From now on, consider symmetric tensors $T$ of degree $7+2q$ in three variables, which satisfy the following assumptions:

\begin{itemize}
\item Fix the integer $q\geq 0$ and let $T$ be a symmetric tensor of degree $d=7+2q$ in $3$ variables, 
with a decomposition $A=\{P_1,\dots,P_r\}\subset\Pj^2$,
of length $r=\ell(A)\leq 10+3q$ such that $A$ is contained in a plane cubic curve $C$. 
\item Assume that the Kruskal's rank of $A$ is $k_1=\min\{3,r\}$ and the 
 $(q+3)$-th Kruskal's rank $k_{q+3}$ of $A$ is equal to $\min\{r,3q+9\}$.
\end{itemize}   

The second assumption implies that no $3$ points of $A$ are aligned.

\begin{remark} Assume that $r=\ell(A)$ is smaller than $3q+10$. Then the second assumption implies that the
$(q+3)$-th Kruskal's rank of $A$ is $r$. 

Take the partition $d=2q+7=(q+3)+(q+3)+1$. Then $r$ satisfies:
$$r \leq \frac {k_{q+3}+k_{q+3}+k_1-2}2,$$
thus by Kruskal's Theorem \ref{K}, we get that $T$ has rank $r$ and $A$ is the unique decomposition
of $T$ (up to trivialities).
\end{remark}

Thus, from now on we will assume that $A$ has cardinality $3q+10$. Notice that, when $r=3q+10$, the Kruskal's 
criterion gives no information (e.g. for the partition $d=2q+7=(q+3)+(q+3)+1$, we have $3q+10> ((3q+9)+(3q+9)+3-2)/2$).
Hence we need to apply another strategy.

\begin{remark}
Assume $r=3q+10$. Since $A$ is contained in a cubic curve, the difference of the Hilbert function of $A$ satisfies
$Dh_A(0)=1$, $Dh_A(1)\leq 2$, $Dh_A(2)\leq 3$ and also $Dh_A(3)\leq 3$.
Thus  $Dh_A(i)\leq 3$ for $i>3$, by Proposition \ref{nonincr}. 

Assume that $Dh_A(1)\leq 1$. Then $h_A(1)<3$, which contradicts the assumption $k_1=3$.
Thus $Dh_A(1)=2$.

Assume that $Dh_A(i)<3$ for some $i$ in the range $2\leq i\leq q+3$. Then:
$$h_A(q+3)=\sum_{j=0}^{q+3}Dh_A(j)\leq 1+2+3(q+1)+2<3q+9,$$
which contradicts the assumption $k_{q+3}=3q+9$. Thus $Dh_A(i)=3$ for $i=2,\dots, q+3$.

It follows that $\sum_{j=0}^{q+3}Dh_A(j)= 3q+9$. Moreover, $Dh_A(q+4)$ can not be $0$,
otherwise $Dh_A(j)=0$ for $j\geq q+4$, by Proposition \ref{nonincr}, hence $\sum_{j=0}^{\infty}Dh_A(j)<3q+10$,
a contradiction. It follows that $Dh_A(q+4)=1$ and  $Dh_A(j)=0$ for $j>q+4$. 

In particular, we get $h_A(2)=6$ and $h_A(3)=9$, which means that $A$ is contained in no conics and in exactly 
one cubic curve.

The function $Dh_A$ is thus the one displayed below.
\begin{center}\begin{tabular}{c|cccccccccc}
$i$           & $0$ & $1$ & $2$ &  $3$ &   \dots &    $q+2$ &  $q+3$ &$q+4$ &$q+5$ &\dots  \\  \hline
$Dh_A(i)$ & $1$ & $2$ & $3$ &  $3$ &   \dots &    $3$     &  $3$     & $1$    &  $0$    &\dots
\end{tabular}\end{center}
\end{remark}
\smallskip

We want to exclude the existence of a second decomposition $B$ of $T$ of length $\ell(B)\leq 3q+10$. So assume,
by contradiction, that $B$ exists, and assume, as above, that $B$ is minimal. Define as above $Z=A\cup B$, so that $\ell(Z)\leq 6q+20$.

Using Lemma \ref{CBdis}, we can prove that the intersection $A\cap B$ can not be empty.

\begin{proposition}\label{CubDis}
The Cayley-Bacharach property $CB(d)$ can not hold for $Z=A\cup B$. Thus $A\cap B\neq \emptyset$.
\end{proposition}
\begin{proof}
Assume that $CB(d)$ holds for $Z=A\cup B$.
Thus, by Theorem \ref{GKRext} we have:
$$\sum_{j=0}^{q+3}Dh_Z(j) \leq \sum_{j=q+5}^{d+1}Dh_Z(j).$$ 
Since $\sum_{j=0}^{q+3}Dh_Z(j)\geq \sum_{j=0}^{q+3}Dh_A(j)=3q+9$, we get:
\begin{multline*} 6q+20\geq \ell(Z)\geq \sum_{j=0}^{2q+8}Dh_Z(j) =\sum_{j=0}^{q+3}Dh_Z(j)+Dh_Z(q+4)+\sum_{j=q+5}^{2q+8}Dh_Z(j)
\\ \geq 3q+9 + Dh_Z(q+4) + 3q+9,\end{multline*}
thus $Dh_Z(q+4)\leq 2$. But then, by Proposition \ref{nonincr}, $Dh_Z(j)\leq 2$ for $j\geq q+4$, thus:
$$3q+9 \leq \sum_{j=q+5}^{2q+8}Dh_Z(j)\leq 2(q+4),$$
a contradiction.

The second claim follows by Lemma \ref{CBdis} applied to $A,B$ and $T$.
\end{proof}

Now we can prove the main results of this section.

 \begin{proposition} The existence of $B$ yields a contradiction.
 \end{proposition}
\begin{proof}
Suppose there is another decomposition $B$ of $T$, of cardinality $\ell(B)=k\leq 10+3q$. 
We know from Proposition \ref{CubDis} that $A\cap B\neq \emptyset$.
Thus we can write, without loss of generality,  $B=\{ P_1,\dots, P_i,P'_{i+1},\dots,P'_r\}$, i.e.
we may assume that $A\cap B=\{ P_1,\dots, P_i\}$, $i>0$.
Then there are coefficients $a_1,\dots,a_{3q+10},b_1,\dots,b_k$ such that:
\begin{multline*}T=a_1v_d( P_1)+\dots+a_iv_d(P_i)+a_{i+1}v_d(P_{i+1})\dots+a_{3q+10} v_d(P_{3q+10}) = \\
b_1 v_d(P_1)+\dots+b_i v_d(P_i)+ b_{i+1}v_d(P'_{i+1})+\dots+b_k v_d(P'_k).\end{multline*}
Consider the tensor 
$$T_0 = (a_1-b_1)v_d( P_1)+\dots+(a_i-b_i)v_d( P_i)+a_{i+1}v_d(P_{i+1})+\dots+a_{3q+10} v_d(P_{3q+10}),$$
which is also equal to $ b_{i+1}v_d(P'_{i+1})+\dots+b_k v_d(P'_k)$. Thus $T_0$ has the two decompositions $A$ and
$B'=\{P'_{i+1},\dots,P'_k\}$, which are disjoint. Thus, if $A$ and $B'$ are both minimal, as decompositions
of $T_0$, then by  Lemma \ref{CBdis} applied to $A,B'$ and $T_0$, we get that $A\cup B'$ satisfies $CB(d)$.
Since $A\cup B'=A\cup B=Z$, and we know by Proposition \ref{CubDis} that $Z$ does not satisfies $CB(d)$,
we find that either $A$ or $B'$ are not minimal.  

Assume that $B'$ is not minimal. Then we can find a point of $B'$, say $P'_k$, such that
$T_0$ belongs to the span of $v_d(B'\setminus\{P'_k\})$. Since $T=T_0+b_1v_d(P_1)+\dots+b_iv_d(P_i)$,
this would mean that $T$ belongs to the span of $v_d(B\setminus\{P'_k\})$, which contradicts the minimality of $B$.

Assume that $A$ is not minimal, and $T_0$ belongs to the span of $v_d(A\setminus\{P_j\})$,
for some $j>i$. That is,  as above, since $T=T_0+b_1v_d(P_1)+\dots+b_iv_d(P_i)$,
this would mean that $T$ belongs to the span of $v_d(A\setminus\{P_j\})$, which contradicts the minimality of $A$.

Assume that $A$ is not minimal, and $T_0$ belongs to the span of $v_d(A\setminus\{P_j\})$,
for some $j\leq i$, say $j=1$. Then $T_0=\gamma_2v_d(P_2)+\dots+\gamma_{3q+10}v_d(P_{3q+10})$,
for some choice of the coefficients $\gamma_j$. Since $v_d(A)$ is linearly independent, because $A$ is minimal for $T$,
this is only possible if $a_1-b_1=0$. So there exists a proper subset $A'\subset A$ which provides a minimal decomposition of
$T_0$, together with $B'$. Moreover $\ell(A')\geq \ell(B')$. Since $A$ has $(q+3)$-th Kruskal's rank $3q+9\geq \ell(A')$, by Remark \ref{ksub} the 
$(q+3)$-th Kruskal's rank $k'_{q+3}$ of $A'$ is $\ell(A')$. Similarly, the Kruskal's rank $k'_1$ of $A'$ is $\min\{3,\ell(A')\}$.
Moreover $\ell(A')>1$,  otherwise $i=3q+9$, $k=3q+10$ and $T_0=a_{3q+10}v_d(P_{3q+10})=b_{3q+10}v_d(P'_{3q+10})$, 
which means that, as projective points,  $P_{3q+10}=P'_{3q+10}$, a contradiction. Thus one computes:
$$ \ell(A')\leq \frac {k'_{q+3}+k'_{q+3}+k'_1-2}2.$$
hence the existence of a second decomposition $B'$ of $T_0$, with $\ell(A')\geq \ell(B')$, contradicts Kruskal's Theorem \ref{K}.
\end{proof}

\begin{remark} We can apply verbatim the previous procedure even for the case $q=-1$. We get ternary forms of 
degree $7+2q=5$ and rank $r\leq 10+3q=7$. In particular, notice that $7$ is the rank of a generic quintic form,
by \cite{AlexHir95}. Notice also that, for $q=-1$, the assumption that $A$ is contained in a cubic is unnecessary,
for any set of cardinality $\leq 7$ in $\Pj^2$ lies in a cubic.
Thus, we get back, from our procedure, the classically well known fact that a general
ternary form $T$ of degree $5$ has a unique decomposition with $7$ powers of linear forms.

Indeed, we can give a more precise notion of {\it generality} for $T$: the uniqueness holds if a decomposition $A$ of $T$
has Kruskal's ranks $k_1=3$ and $k_2=6$.
\end{remark}

\bibliographystyle{amsplain}
\bibliography{biblioLuca}
\end{document}